\documentclass{amsart}
\usepackage{graphicx} % Required for inserting images
\usepackage{amsthm}
\usepackage{lipsum}
\usepackage[utf8]{inputenc}
\usepackage{amsfonts}
\usepackage{amsmath}
\usepackage{multicol, caption}
\usepackage{amsthm,dsfont}
\usepackage{amssymb}
\usepackage{fancyhdr}
\usepackage{emptypage}
\usepackage{subfiles}
\usepackage{tasks}
\usepackage{multirow}
\usepackage{nicefrac}
\usepackage{times}
\usepackage{enumerate}
\usepackage{array}
\usepackage{float}
\usepackage{colortbl}
\usepackage{enumitem,ulem}
\usepackage{wrapfig}
\usepackage{comment}
\usepackage{tkz-euclide}
\usepackage{imakeidx}
\newtheorem{definicion}{Definition}
\newtheorem{theorem}{Theorem}

\newtheorem{lema}{Lemma}

\newtheorem{remark}{Remark}

\DeclareMathOperator{\diver}{div}

\DeclareMathOperator{\ds}{ds}
\DeclareMathOperator{\dt}{dt}
\DeclareMathOperator{\dr}{dr}

\DeclareMathOperator{\D}{d}

\newcommand{\R}{\mathbb{R}}

\title[Nonexistence of positive radial solutions for semipositone $\phi$-Laplacian]{Nonexistence of positive radial solutions for semipositone $\phi$-Laplacian problems with superlinear reaction term}
\author{Sigifredo Herr\'on, \ Emer Lopera and Diana S\'anchez}
\address{Sigifredo Herr\'on \newline
 Departamento de Matem\'aticas,
 Universidad Nacional de Colombia Sede Medell\'in,
 Medell\'in, Colombia}
\email{sherron@unal.edu.co}
\address{Emer lopera \newline
 Departamento de Matem\'aticas,
 Universidad Nacional de Colombia Sede Manizales,
 Ma\-ni\-zales, Colombia}
\email{edloperar@unal.edu.co}
\address{Diana S\'anchez \newline
 Departamento de Matem\'aticas,
 Universidad Nacional de Colombia Sede Manizales,
 Ma\-ni\-zales, Colombia}
\email{dmsanchezm@unal.edu.co}
\date{}

\begin{document}
\maketitle
\begin{center}
\textit{    This work is dedicated to Professor Shivaji, in recognition of his significant contributions to the field.}
\end{center}
\begin{abstract}
\noindent
The aim of this paper is to prove the nonexistence of positive radial solutions to the problem $-\Delta_\phi u=\lambda f(u)$, $x\in B_1(0)$, $u(x)=0$ on $|x|=1$, for  $\lambda>0$ sufficiently large. Here, $\phi$ is a continuous function, $\Delta_\phi$ denotes the $\phi$-Laplacian operator  which is defined by $\Delta_\phi (u):=\diver (\phi (|\nabla u|) \nabla u)$, and $B_1(0)$ is the unit ball in $\mathbb{R}^N$, with $N>1$. Furthermore, $f$ is a continuous, nondecreasing function such that $f(0)<0$, and its behavior at infinity is intimately related to $\phi$. Our findings are presented in a combined format, employing both an indirect argument and an energy analysis.
\end{abstract}


%%%%%%%%%%%%%%%%%%%
\section{Introduction}
We study the nonexistence of positive solutions for the problems of the form \begin{equation}\label{prob-Cauchy}
		\left\{
		\begin{array}{rcl}
			-\Delta_\phi u &=&\lambda f (u) , \quad x\in B_1(0), \\ 
			u(x)&=&0 , \qquad \quad x\in \partial B_1(0),
		\end{array} 
		\right. 
	\end{equation}
	where  $\phi:\R \to [0,\infty)$ is a continuous function, differentiable in $\mathbb{R}\smallsetminus\{0\}$,   $\lambda >0$, $B_1(0)\subset \R^N$ is the unit ball, $N > 1$. Here $\Delta_\phi$ denotes the  $\phi$-Laplacian operator which is defined by $\Delta_\phi u:=\diver (\phi (|\nabla u|) \nabla u)$. The reaction term, $f$, is a continuous function on $[0, \infty)$ that satisfies additional conditions, which will be presented below.
     From now on, \(c_1, c_2, \ldots, C, C_1, C_2, \ldots\), denote positive constants independent of the solutions that may vary from line to line.

 We assume:
 \begin{itemize}
		\item [$(\boldsymbol{f_1})$] 
	$f\colon [0,\infty)\to\R$ is a continuous function,  with only one zero $u_0>0$, $f $ is nondecreasing and $f(0)<0$ (semipositone).
    \item [$(\boldsymbol{f_2})$] There exist  $p> 1$ and $\alpha\in (p-1,p^{*}) $ such that $\lim_{s\to \infty} \frac{f(s)} {s^{\alpha}} >0$, where $p^*$ is the critical exponent defined by $p^{*}=Np/(N-p)$ for $N>p$ and $p^*=\infty$ for $p\geqslant N$.
  \item [$(\boldsymbol{\varphi})$] The function $\phi$ is increasing in $(0, \infty)$ and even. If we define $\varphi:\R\to \R$ as \begin{equation}\label{varphi-phi}
    \varphi (s):=\phi (s)s,
\end{equation}  
 let us assume that there exist positive  constants $\hat{c}_1$, $\hat{c}_2$  such that \begin{equation}\label{hipot_phi_2}
			\hat{c}_2r ^{p-1} \leqslant \varphi (r) \leqslant \hat{c}_1r ^{p-1}, \quad \text{for all }\ r\geqslant 0 .
		\end{equation}
  	  \end{itemize}
Observe that $\varphi$ is increasing and odd. We also note that \eqref{hipot_phi_2} is equivalent to
\begin{eqnarray}\label{estim-varphi-invers}
 \hat{c}_1^{1/(1-p)}t^{1/(p-1)}\leqslant\varphi^{-1}(t)\leqslant \hat{c}_2^{1/(1-p)}t^{1/(p-1)}, \quad     \forall\,t\geqslant0.
\end{eqnarray}
 This implies, for $s\geqslant0$ and $t\geqslant0$,
 \begin{eqnarray}\label{varphi-inv-product}
     \left(\hat{c}_1^{-1}\hat{c}_2^2\right)^{1/(p-1)}\varphi^{-1}(s)\varphi^{-1}(t)\leqslant \varphi^{-1}(st)\leqslant \left(\hat{c}_2^{-1}\hat{c}_1^2\right)^{1/(p-1)}\varphi^{-1}(s)\varphi^{-1}(t).
 \end{eqnarray}
 
 Moreover, from $\boldsymbol{(f_2)}$ there exists a constant $k>0$ such that
 \begin{eqnarray}\label{fest^alfa}
     f(t)\geqslant kt^\alpha, \quad \forall\, t\geqslant\frac{u_0+U_0}{2},
 \end{eqnarray}	
 where $U_0>0$ is the unique zero of $F$ (clearly, $U_0>u_0$), with $F(u):=\int_{0}^{u}f(s)\ds$ (see Figure \ref{fandF}). Assertion \eqref{fest^alfa} can be established, for example, by using the definition of the limit and that $s\mapsto f(s)/s^\alpha$ is continuous on compact sets from away the origin.

\begin{figure}[h]
\centering
\includegraphics[width=8cm]{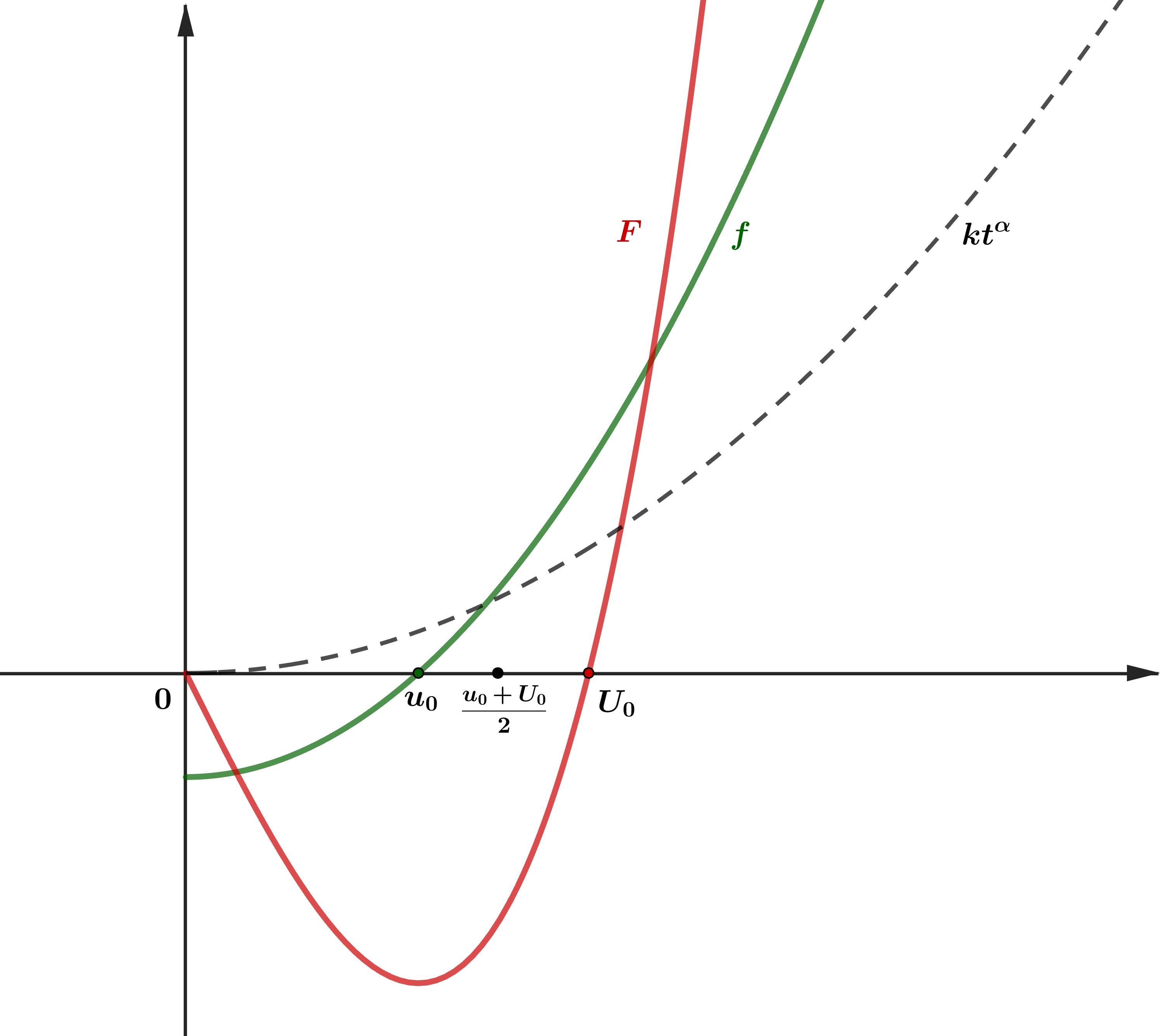}
\caption{Graphs of $f$ and $F$}\label{fandF}
\end{figure}

The primary result of this study is as follows.

\begin{theorem}\label{no-positive}
    Under hypotheses $(\boldsymbol{f_1}), (\boldsymbol{f_2})$ and    $(\boldsymbol{\varphi})$  there exists $\lambda _0>0$ such that for $\lambda>\lambda_0$, the problem \eqref{prob-Cauchy} has no positive radially symmetric solution.
\end{theorem}
The nonexistence of positive solutions for Dirichlet problems like \eqref{prob-Cauchy} for the $p$-Laplacian, has been a topic of study for several years. Initial research was conducted in the positone case ($f(0) > 0$). For further information,  see \cite{N} and references therein on the semilinear case and \cite{Bonanno, hai-shivaji} in the quasilinear case ($p$-Laplacian). To the best of our knowledge, this is the first result about nonexistence concerning problem \eqref{prob-Cauchy} in the semipositone context. 
Consequently, we are confronted with a dual challenge, namely that the differential operator is more general than the $p$-Laplacian, and that the nonlinearity exhibits semipositone characteristics. 
Recently, Abbebe and Chhetri in \cite{AC} addressed the nonexistence problem for a system superlinear, subcritical in the semipositone case. 
\\
 Our procedure is based on an energy analysis, which involves computing the derivative of the energy function (see equation \eqref{def_energy1}), this allows us to obtain some monotonicity. However, since our solutions do not possess second-order derivatives at all points, we are compelled to compute the derivative using the very definition of the energy function. This was one of the challenges encountered during this work.  Our theorem extends a result obtained in \cite{CG}, where the authors considered a similar problem involving the differential operator  $p$-Laplacian. Other results on nonexistence can be found in \cite{hai, HL}.\\
Furthermore, with this work, we intend to continue the study of the problem \eqref{prob-Cauchy} for any positive value of the parameter $\lambda$. Indeed, it is known that for $\lambda>0$ small, problem \eqref{prob-Cauchy} has at least one positive solution, under some additional restriction on $\varphi$ (see \cite{HLS}).\\ 
This paper is structured as follows. Section 2 defines the notions of solution and energy associated with the problem. It also establishes the monotonicity properties of this energy. Section 3 is dedicated to the proof of the main result. In this section, the behavior of potential solutions of problem \eqref{prob-Cauchy} is analyzed, and the results are concluded at the end of this section.

%%%%%%%%%%%%%%%
\section{Preliminaries}
It is well-known that the radial version corresponding to  \eqref{prob-Cauchy} is given by
	\begin{equation}\label{prob-Cauchy-radial}
		\left\{
		\begin{array}{rll}
			\left( r^{N-1} \varphi (u')\right)'  &=-r^{N-1}\lambda f (u),  & 0<r<1, \\ 
			u'(0)&=0, & \\
			u(1)&=0, &
		\end{array} 
		\right. 
	\end{equation}
	where $u^\prime$ denotes the derivative of $u$ with respect to $r$. In addition,  a solution of the ode in \eqref{prob-Cauchy-radial} satisfies		\begin{equation}\label{form_int_u'} 		u' (r)=- \varphi ^{-1}\left( r^{1-N}\int_{0}^r \lambda t^{N-1} f (u)\dt \right),  \quad 0<r \leqslant 1,\end{equation}
	and thus  the differential equation can be written as
 \begin{equation}\label{edradial}
   (\varphi(u{'})){'}+\frac{N-1}{r}\varphi(u{'})=-\lambda f(u).  
 \end{equation}
  \begin{definicion}
     We say that $u$ is a solution of the the problem \eqref{prob-Cauchy-radial} if $u \in C^1([0,1])$, $\varphi (u') \in C^1((0, 1))$, $u$ satisfies the differential equation pointwise $u'(0)=0$ and $u(1)=0$.
 \end{definicion}

 Let us define  $ \Phi (t):=\int_0^{t} \varphi (s) \ds $. For any $u$, solution of problem \eqref{prob-Cauchy-radial},
 we define the energy associated with our problem as \begin{equation}\label{def_energy1}
		E(r)=\lambda F(u) +\varphi (u')u'-\Phi (u').
	\end{equation} 

 \begin{remark}{\label{deri_energia}}
       $E{'}(r)\leqslant 0$, for all $r>0$. Indeed, 
       for  $r>0$ where $u{'}(r)\neq 0,$ we have
     \begin{align*}
         E'(r)&=\lambda f(u)u{'}+(\varphi(u{'})){'}u{'}+\varphi(u{'})u{''}-\varphi(u{'})u{''}\\
         &=(\lambda f(u)+(\varphi(u{'})){'}))u{'}\\
         &=u{'}\left(-\frac{N-1}{r}\varphi(u{'})\right)\\ &=-\frac{N-1}{r}(u{'})^2\phi(u{'}).
         \end{align*}
         The last equality arises from  \eqref{edradial} and the fact that $s\varphi(s)=s^2\phi(s)$. On the other hand, for $r>0$ such that $u{'}(r)= 0,$ in light of the fact that  $u''(r)$ does not necessarily exist, we shall proceed as follows. First, we note that
\begin{align*}
    \left|\frac{\Phi(u{'}(r+h))}{h}\right|&=\frac{1}{|h|}\left|\int_0^{|u{'}(r+h))|} \varphi (s) \ds\right|\\
    &\leqslant \frac{\varphi(|u{'}(r+h))|)|u{'}(r+h))|}{|h|}\\
    &=\frac{|u{'}(r+h)|(|\varphi(u{'}(h+r))-\varphi (u{'}(r))|)}{|h|}.
\end{align*}
Therefore
\begin{align*}
    E{'}(r)&=\lim_{h \to 0}\frac{E(r+h)-E(r)}{h}\\ 
    &=\lim_{h \to 0} \left[\frac{\lambda(F(u(r+h)-F(u(r)))}{h}+\frac{u{'}(r+h)(\varphi(u{'}(h+r))-\varphi (u{'}(r)))}{h}\right.\\
    &\left. \hspace{2cm} -\frac{\Phi(u{'}(r+h))}{h}\right]\\ &=\lambda f(u)u{'}(r)=0.
\end{align*}
In both cases we obtain
\begin{equation}\label{derivada_energia}
        E'(r)=-\frac{N-1}{r}\varphi(u{'}(r))u{'}(r)\leqslant 0 \quad \forall r\in (0,1].
    \end{equation}
\end{remark}

\begin{remark}{\label{signo_var-phi}}
  For $t\geqslant 0$ we have $ \Phi (t)=\int_0^{t} \varphi (s) \ds \leqslant t\varphi (t)$.   Thus, since $\Phi$ is even then
    $\varphi(s)s-\Phi(s)\geqslant 0$. Therefore, $ E(1)\geqslant0$. Moreover, 
    since $E'(r) \leqslant 0 \mbox{ for all }r\in [0,1]$, by Remark \ref{deri_energia}, then $E(r)\geqslant 0$ for all $r\in [0,1]$.\end{remark}

%%%%%%%%%%%%
\section{Proof of the main result}
\begin{lema}{\label{lemma1}}
    If $\lambda >0$ and $u$ is a positive solution of \eqref{prob-Cauchy-radial} then $u(0)\geqslant U_0$ (see Figure \ref{u>u_0}).  
\end{lema}
\begin{proof}
    By Remark \ref{deri_energia}, Remark \ref{signo_var-phi} and the fact that $E$ is a continuous function, we have 
    \[0\leqslant E(1)\leqslant E(0)=\lambda F(u(0)).\]
    Hence $u(0)\geqslant U_0$.
    \end{proof}

\begin{figure}[h]
\centering
\includegraphics[width=6cm]{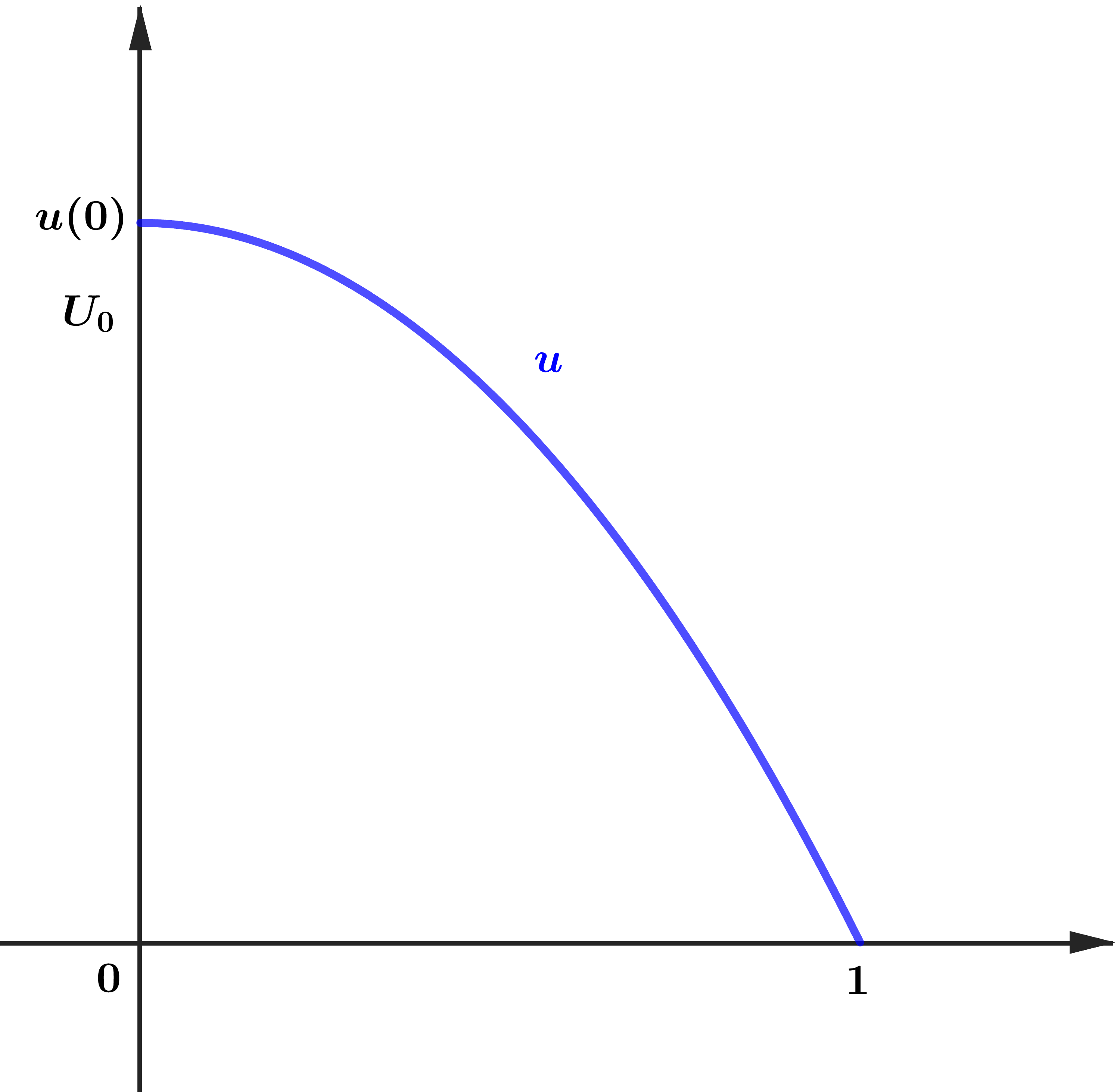}
\caption{Graph of $u$}\label{u>u_0}
\end{figure}
    
    \begin{lema}{\label{lemma2}}
        There exists $\lambda_1>0$ such that if $\lambda>\lambda_1$ and $u$ is a positive solution of \eqref{prob-Cauchy-radial}, then there exists $r_1:=r_1(\lambda)\in (0,1/2]$ such that $u(r_1)=(u_0+U_0)/2.$
    \end{lema}
    \begin{proof}
        Let us assume, by contradiction, that there exist $\lambda_n \to \infty$ when $n\to \infty$ such that $u_{\lambda_n}(r)\neq (u_0+U_0)/2,$ for all $r\in [0,1/2]$. Moreover, in view of the Lemma \ref{lemma1} and the continuity of $u_{\lambda_n}$, we can assume without loss of generality that $u_{\lambda_n}(r) > (u_0+U_0)/2,$ for all $r\in [0, 1/2]$.   Using \eqref{fest^alfa} and \eqref{form_int_u'} we get $u_{\lambda_n}^\prime (s)<0$ for all $s\in [0,r]$. Therefore, for all $r\in [0,1/2]$ we have
        \begin{align*}
            r^{N-1}\varphi(u'_{\lambda_n})&=-\lambda_n \int_{0}^{r}s^{N-1}f(u_{\lambda_n}(s))\ds\\
            &\leqslant -\lambda_n\int_{0}^{r}s^{N-1}ku_{\lambda_n}^{\alpha}(s)\ds\\
            &=\frac{-\lambda_n k}{N}\int_{0}^{r}(s^{N}){'}ku_{\lambda_n}^{\alpha}(s)\ds\\
            & \leqslant \frac{-\lambda_n k}{N}\left(r^{N}u_{\lambda_n}^{\alpha}(r)-\int_{0}^{r}\alpha s^{N}u_{\lambda_n}^{\alpha-1}(s)u'_{\lambda_n}(s)\ds \right)\\
            & \leqslant \frac{-\lambda_n k}{N}r^{N}u_{\lambda_n}^{\alpha}(r).
        \end{align*}
        Therefore, due to \eqref{estim-varphi-invers},
        \begin{align}\label{u'<u}
            u{'}_{\lambda_n}(r) &\leqslant -\varphi^{-1}\left(\frac{\lambda_n k}{N}ru_{\lambda_n}^{\alpha}(r)\right) \nonumber\\
            &\leqslant -C\left(\frac{\lambda_n k}{N}ru_{\lambda_n}^{\alpha}(r)\right)^{1/(p-1)} \nonumber\\ 
            &= \frac{-C\lambda_n^{1/(p-1)}k^{1/(p-1)}r^{1/(p-1)}u_{\lambda_n}^{\alpha/(p-1)}}{N^{1/(p-1)}}.   
            \end{align}    
     Organizing and then integrating  the inequality \eqref{u'<u}  we have
    \begin{align*}
    \int_{u_{\lambda_n}(0)}^{u_{\lambda_n}(1/4)}z^{-\alpha/(p-1)}\D z &\leqslant \int_{0}^{1/4}-C\lambda_n^{1/(p-1)}r^{1/(p-1)}\dr
    \end{align*}
    and then
    \begin{equation*}
         \frac{(u_{\lambda_n}(1/4))^{1-\alpha/(p-1)}-(u_{\lambda_n}(0))^{1-\alpha/(p-1)}}{1-\alpha/(p-1)}\leqslant -C\lambda_n^{1/(p-1)}.
    \end{equation*}
    Hence
    \begin{align*}
    (u_{\lambda_n}(1/4))^{1-\alpha/(p-1)}&\geqslant C(\alpha/(p-1)-1)\lambda_n^{1/(p-1)}.
      \end{align*}
      Equivalently
\begin{equation}\label{u>lam}
        \frac{1}{u_{\lambda_n}(1/4)}\geqslant C\lambda_n^{\frac{1}{\alpha-(p-1)}}.
\end{equation}
      Since the right hand side of \eqref{u>lam} goes to infinity
      as $n\to\infty$, then $u_{\lambda_n}(1/4)\to 0$ as $n\to \infty$.
    This is a contradiction, as  $u_{\lambda_n}(1/4)>(u_0+U_0)/2$. This proves the lemma.
    \end{proof}
    \begin{lema}{\label{lema3}}
There exists $\lambda_2>0$ such that if $\lambda>\lambda_2$ and $u$ is a positive solution of \eqref{prob-Cauchy-radial}, then there exists $r_2:=r_2(\lambda)\in [3/4,1)$ such that $u(r_2)=u_0/c$ for any $c>2$.
    \end{lema}
 \begin{proof}
     Suppose, by contradiction, there exists $\{\lambda_n\}_{n=1}^{\infty}$, such that $\lambda_n\to \infty$, and a corresponding sequence of positive solutions $\{u_{\lambda_n}\}_{n=1}^{\infty}$ of \eqref{prob-Cauchy-radial} and a sequence of $\{c_n\}_{n=1}^{\infty}$, with $c_n>2$, such that $u_{\lambda_n} (r)\neq u_0/c_n$ for all $r\in [3/4,1)$. Using the fact that $u_{\lambda_n}(1)=0$ and the continuity of $u_{\lambda_n}$, we infer that  $u_{\lambda_n}(r)<u_0/c_n$ for all $r\in [3/4,1)$. Therefore, $u_{\lambda_n}(r)<u_0/2$ for all $r\in[3/4,1)$ since $c_n>2$.\\
     Integrating between $r$ and $1$ the differential equation  in \eqref{prob-Cauchy-radial} with $r\in [3/4,1)$ and taking into account that $f$ is nondecreasing, we get
\begin{align*}
r^{N-1}\varphi(u'_{\lambda_n})&=\int_r^1 \lambda_n s^{N-1}f(u_{\lambda_n})\ds+\varphi(u'_{\lambda_n}(1))\\
&\leqslant \lambda_n \int_r^1
s^{N-1}f(u_{\lambda_n})\ds\\ %&\leqslant \lambda_n f(u_0/2)\int _r^1 s^{N-1}\ds\\
&\leqslant\lambda_n f(u_0/2)\left(\frac{1}{N}-\frac{r^{N}}{N}\right).\end{align*} 
Since $f(u_0/2)<0$, we obtain
\begin{align*}
\varphi(u{'}_{\lambda_n})&\leqslant \frac{\lambda_n}{N}\left(\frac{1}{r^{N-1}}-r\right)f(u_0/2)\\&\leqslant\frac{\lambda_n}{N}(1-r)f(u_0/2)\equiv-C\lambda_n(1-r).
%-\varphi(u{'}_{\lambda_n})&\geqslant -\frac{\lambda_n}{N}(1-r)f(u_0/2).
\end{align*}
Then from \eqref{estim-varphi-invers}, $-u_{\lambda_n}'\geqslant C\left(\lambda_n(1-r)\right)^{1/(p-1)}=C\lambda_n^{1/(p-1)}(1-r)^{1/(p-1)}$. Hence 
\begin{align*}
    u_{\lambda_n}(r)&=-\int_r^1u_{\lambda_n}'(t)\dt \geqslant \int_r^1C\lambda_n^{1/(p-1)}(1-t)^{1/(p-1)}\dt\\ %&=C\lambda_n^{1/p}\int_0^{1-r}t^{1/p}\dt\\
    &=C\lambda_n^{1/(p-1)}(1-r)^{p/(p-1)}.
    \end{align*}
    Taking  $r=4/5$ we have \[u_{\lambda_n}(4/5)\geqslant C\lambda_n^{1/(p-1)}\to \infty,\]
    as $n\to \infty$. This is impossible because $u_{\lambda_n}(4/5)<u_0/2.$
 \end{proof}
 We are ready to prove the main result of this section.
 \begin{proof}[Proof of Theorem \ref{no-positive}]
     Let us assume that the theorem is not true. Therefore, there exists a sequence of positive numbers $\{\lambda_n\}_{n=1}^{\infty}$, $\lambda_n \to \infty$, as $n\to \infty$ such that for each $\lambda_n$, there exists a positive solution $u_{\lambda_n}$ of \eqref{prob-Cauchy-radial}.  We claim that $u'_{\lambda_n}(r)<0$ for each $r\in (0,1)$ and $u'_{\lambda_n}(1)\leqslant 0$. Indeed, there exists $t_0\in (0, 1]$ such that $u(t_0)=u_0$. Without loss of generality, let us take $t_0$ as the smallest with this property. Then from \eqref{form_int_u'}, $u_{\lambda_n}'(t)<0$ for all $t\in (0, t_0]$. On the other hand, for $r\in [t_0,1],$ we also have  $u'_{\lambda_n}(r)<0$. In fact, let $r_0:=\min \{t\in [t_0, 1]: u'(t)\geqslant 0\}$. If $r_0<1$, then there exists $t_1\in (t_0, r_0]$ such that $u_{\lambda_n}'(t_1)=0$ and $u_{\lambda_n}(t_1)\leqslant u_0$. Then $E(t_1)<0$, which is a contradiction. Therefore $r_0=1$. That is, $u'_{\lambda_n}(r)<0$ for all $r\in (t_0, 1)$. Moreover, since $u'$ is continuous, $u'_{\lambda_n}(1)\leqslant 0$. This proves the claim. \\
     For every $n$, let $\lambda_1^*>0, r_1=r_1(\lambda_n)$ and $\lambda_2^*>0,  r_2=r_2(\lambda_n)$ be as in Lemmas \ref{lemma2} and \ref{lema3},
     respectively.\\ Let us fix $n$, such that $\lambda_n\geqslant \max\{\lambda_1^*,\lambda_2^*\}$.  By the mean value theorem there exists $r_3\in(r_1,r_2)$ such that
     \begin{align*}
         |u{'}_{\lambda_n}(r_3)|= \frac{|u_{\lambda_n}(r_2)-u_{\lambda_n}(r_1)|}{r_2-r_1}  \leqslant 
         4|u_0/c-(u_0+U_0)/2|\leqslant 6U_0,
         \end{align*}         
         since $c>2$ and $u_0<U_0$.  
         Letting  $M:=\max\{F(u)\colon u_0/c \leqslant u \leqslant (u_0+U_0)/2\} <0$ and using the assumption $\boldsymbol{(\varphi)}$, $u_{\lambda_n} $ is decreasing in $(r_1, r_2)$, $\Phi\geqslant0$ and $\varphi$ is odd,    
         \begin{eqnarray*}
             E(r_3)&=&\lambda_n F(u_{\lambda_n}(r_3)) +\varphi (u'_{\lambda_n}(r_3))u'_{\lambda_n}(r_3)-\Phi (u'_{\lambda_n}(r_3))\\
             &\leqslant & \lambda_n M+\hat{c}_1|u'_{\lambda_n}(r_3)|^{p}\\
             &\leqslant & \lambda_n M+\hat{c}_1(6U_0)^p.
         \end{eqnarray*}  
         Therefore,
         \[\lim_{n\to \infty} E(r_3)=-\infty ,\]         
which contradicts the fact that $E\geqslant 0$ (Remark \ref{signo_var-phi}) and this proves the theorem.
 \end{proof}

%%%%%%%%%%%%%%%%%
\section*{Acknowledgments}
We would like to express our gratitude to Ratnasingham Shivaji for encouraging us to address this problem. Diana Sánchez would like to express her gratitude to Maya Chhetri for providing her with some references that were essential in the development of this work. Authors Sigifredo Herr\'on and Diana S\'anchez were partially supported by Facultad de Ciencias, Universidad Nacional de Colombia, Sede Medell\'\i n– Facultad de Ciencias – Departamento de Matem\'aticas – Grupo de investigaci\'on en Matem\'aticas de la Universidad Nacional de Colombia Sede Medell\'\i n. Proyecto de facultad:   An\'alisis no li\-ne\-al aplicado a problemas mixtos en ecuaciones diferenciales parciales,  c\'odigo Hermes 60827.  
Emer Lopera was supported by Facultad de Ciencias Exactas y Naturales, Universidad Nacional de Colombia, Sede Manizales – Facultad de Ciencias – Departamento de Matemáticas – Grupo de investigación: Análisis Matemático AM de la Universidad Nacional de Colombia-Sede Manizales. Project: Problemas en ecuaciones elípticas no lineales, HERMES code 63271.


\begin{thebibliography}{99}


\bibitem{AC} Abebe, A., $\&$ Chhetri, M. (2023). A nonexistence result for $p$-Laplacian systems in a ball. Electronic Journal of Differential Equations, (Special Issue 02), 1-10.


\bibitem{hai-shivaji}  Alotaibi, T., Hai, D.D. $\&$ Shivaji, R. Existence and nonexistence of positive radial solutions for a class of 
-Laplacian superlinear problems with nonlinear boundary conditions. Communications on Pure and Applied Analysis, 2020, 19(9): 4655-4666. doi: 10.3934/cpaa.2020131


 \bibitem{Bonanno} Bonanno, G.,  Candito, P.,  Livrea, R. $\&$  Papageorgiou, N S. Existence, nonexistence and uniqueness of positive solutions for nonlinear eigenvalue problems. Communications on Pure and Applied Analysis, 2017, 16(4): 1169-1188. doi: 10.3934/cpaa.2017057
 
\bibitem{CG} Chhetri, M., $\&$ Girg, P. (2006). Nonexistence of nonnegative solutions for a class of $(p-1)$-superhomogeneous semipositone problems. Journal of mathematical analysis and applications, 322(2), 957-963.

\bibitem {hai}  Hai, D.D. Nonexistence of positive solutions for a class of p-Laplacian boundary value problems, Applied Mathematics Letters, 31 (2014), pp. 12-15.

\bibitem{HL}Herrón, S., $\&$ Lopera, E. (2015). Non-existence of positive radial solution for semipositone weighted p-Laplacian problems. Electronic Journal of Differential Equations, 2015(130), 1-9.

\bibitem{HLS} Herrón, S., Lopera, E., $\&$ Sánchez, D. (2022). Positive solutions for a semipositone $\phi$-Laplacian problem. Journal of Mathematical Analysis and Applications, 510(2), 126042.


\bibitem{N} Naito, Y. (2000). Nonexistence results of positive solutions for semilinear elliptic equations in $\mathbb{R}^n$. Journal of the Mathematical Society of Japan, 52(3), 637-644.



\end{thebibliography}
 \end{document}